%% file: main.tex
\documentclass[a4paper,american,11pt, reqno]{amsart}
\usepackage{style}

\begin{document}
\title{Schubert coefficients of sparse paving matroids}
\author[Schubert coefficients of sparse paving matroids]{Jon Pål Hamre}

\begin{abstract}
The Chow class of the closure of the torus orbit of a point in a Grassmannian only depends on the matroid associated to the point. The Chow class can be extended to a matroid invariant of arbitrary matroids. We call the coefficients appearing in the expansion of the Chow class in the Schubert basis the Schubert coefficients of the matroid. These Schubert coefficients are conjectured by Berget and Fink to be non-negative. We compute the Schubert coefficients of a disconnected matroid in terms of the Schubert coefficients of its connected components. And we compute the Schubert coefficients for all sparse paving matroids, and confirm their non-negativity.
\end{abstract}

\maketitle


\input{Section1_Introduction}
\input{Section2_Preliminaries}
\input{Section3_Schubert_Coefficients}

\input{Section4_disconnected}
\input{Section5_Sparse_Paving}
\input{Section6_Acknowledgments}

\printbibliography

\end{document}

%% file: Section1_Introduction.tex
\section{Introduction}

To a point $x$ in a complex Grassmannian $G(r,n)$ we associate two objects. The matroid $M_x$ whose bases are given by the non-zero Plücker coordinates of $x$ and the projective, toric subvariety of $G(r,n)$ given by the closure of the orbit of $x$ under the action of the algebraic torus $T= (\CC^*)^n$, denoted $\T$. The class of $\T$ in the Chow ring of $G(r,n)$ is an additive, valuative matroid invariant. Using results of Fink, Derksen, Speyer and Berget in \cite{DerksenFink}, \cite{FinkSpeyer} and \cite{BergetFink} we consider the unique valuative extension of this invariant to all matroids, not just those representable over $\CC$. We denote it by $\Sc(M) \in A^\bullet(G(r,n))$. The Chow ring $A^\bullet(G(r,n))$ is generated as a free abelian group by the Schubert cycles $\sigma_\lambda$ where $\lambda$ ranges over all partitions whose Young diagram is contained in the $r\times (n-r)$ rectangle. Hence $\Sc(M)$ can be expressed as the sum $\Sc(M) = \sum_\lambda d_\lambda(M) \sigma_\lambda$. We call the integers $d_\lambda(M)$ the Schubert coefficients of $M$.

The non-negativity of the Schubert coefficients was conjectured by Berget and Fink in \cite{BergetFink}. We explain how to compute the Schubert coefficients of a disconnected matroid in terms of the Schubert coefficients of its connected components in \cref{thm:disconnected}. And we compute the Schubert coefficients of sparse paving matroids and confirm their non-negativity, supporting the conjecture of \cite{BergetFink}.
\begin{thm}
    \label{thm:1}
    Let $M$ be a connected sparse paving matroid of rank $r$ on $[n]$ then $d_\lambda(M) = d_\lambda(U_{r,n})$ for all partitions $\lambda$ except for the complement of the hook partition $h^c$ for which $d_{h^c}(M) = \beta(M)$.
\end{thm}
\begin{corollary}
    \label{cor:1}
    Let $M$ be a connected sparse paving matroid. The Schubert coefficients $d_\lambda(M)$ are non-negative.
\end{corollary}
We suspect that sparse paving matroids are the only matroids such that $d_\lambda(M) = d_\lambda(U_{r,n})$ for all partitions $\lambda \neq h^c$, and make the following conjecture.
\begin{conj}
    \label{conj:sparsePaving}
    Let $M$ be a connected matroid of rank $r$ on $[n]$. $M$ is sparse paving if and only if $d_\lambda(M) = d_\lambda(U_{r,n})$ for all partitions $\lambda \neq h^c$.
\end{conj}

Section 2 consists of preliminaries on Grassmannians, matroids and valuations. Section 3 introduces the concept of Schubert coefficients, and recaps what is previously known. Section 4 is about Schubert coefficients of disconnected matroids. Section 5 contains the proof of \cref{thm:1} and some evidence for \cref{conj:sparsePaving}.

%% file: Section2_Preliminaries.tex
\section{Preliminaries}

\subsection{Grassmannians}
We denote the complex Grassmannian of sub-vector spaces of $\CC^n$ of dimension $r$ by $G(r,n)$. It is a smooth projective variety embedded into $\PP^{\binom{n}{r}-1}$ by the well known Plücker embedding; see \cite{EisHar}. For a point \(x\in G(r,n)\) we denote the corresponding \(r\)-dimensional subspace of \(\CC^n\) by \(L(x)\). The algebraic torus $T = (\CC^*)^n$ acts on $G(r,n)$ by dilating the coordinates of $\CC^n$. We will be interested in the closure in $G(r,n)$ of the $T$ orbit of $x \in G(r,n)$ denoted $\T$.

\subsection{Partitions}
\begin{definition}
    A partition $\lambda$ is a tuple $\lambda = (\lambda_1, \dots , \lambda_k)$ of positive integers with $\lambda_1 \geq \lambda_2 \geq \dots \geq \lambda_k$.
\end{definition}
We think of partitions in terms of their Young diagrams (sometimes called Ferrers diagrams), i.e.\ a collection of boxes with $\lambda_i$ boxes in row $i$. We write $|\lambda| = \sum_i\lambda_i$ for the number of boxes. There is a partial ordering of partitions defined by $\lambda \subseteq \mu$ if the Young diagram of $\lambda$ is contained in the Young diagram of $\mu$. We will be interested in the partitions $\lambda \subseteq r\times(n-r)$ where $r\times (n-r)$ is the partition whose Young diagram is the $r \times (n-r)$ rectangle. For each $\lambda \subseteq r\times (n-r)$ let $\lambda^c$ be the partition given by the complement of the Young diagram of $\lambda$ in the $r \times (n-r)$ rectangle, rotated by 180 degrees. We denote the so called hook partition by $h$, i.e.\ $h = (n-r,1,\dots,1)$ with $r-1$ ones. Notice that $h^c = (r-1)\times(n-r-1)$. The Young diagrams of $h$ and $h^c$ for $r=3$ and $n=7$ are shown below.
\begin{equation*}
        \scalebox{0.8}{\young(~~~~,~,~)} \hspace{2cm} \scalebox{0.8}{\young(~~~,~~~)}
    \end{equation*}
\begin{definition}
    A Young tableau of shape $\lambda$ is a filling of the boxes of the Young diagram of $\lambda$ with the integers $1,2,\dots, |\lambda|$. A Young tableau is called standard if it is increasing in rows and columns.
\end{definition}
The combinatorics of partitions and Young tableaux appears in many areas of mathematics, such as in the description of the Chow ring of Grassmannians.

\subsection{Chow ring of Grassmannians}
For a partition \(\lambda \subseteq r\times (n-r)\) the associated jumping sequence \(J = (j_i)_{i = 1}^r\) is defined by \(j_i = n-r+i-\lambda_i\). To obtain \(J\) from \(\lambda\), label the steps on the path from \((n-r,r)\) to \((0,0)\) along the Young diagram of \(\lambda\) with the numbers \(1,2,\dots , n\) in order. Then \(j_i\) is the label of the \(i\)th vertical step. The Schubert variety associated to \(\lambda\) and a complete flag \(V_\bullet\) of \(\CC^n\) is defined as 
\[
    X_\lambda(V_\bullet) = \left\{ x \in G(r,n) \mid \dim(L(x) \cap V_{j_i}) \geq i \right\}.
\]
Notice that if \(j_{i+1} = j_{i} + 1\) then
\[
    \dim(L(x) \cap V_{j_{i+1}}) \geq i+1 \implies \dim(L(x) \cap V_{j_i}) \geq i.
\]
So in this situation we may remove the \(i\)th defining inequality of \(X_\lambda(V_\bullet)\). The Chow ring $A^\bullet(G(r,n))$ is generated as a free abelian group by the Schubert cycles $\sigma_\lambda := [X_\lambda(V_\bullet)]$ where $\lambda$ ranges over all partitions $\lambda \subseteq r\times(n-r)$. Here $\sigma_\lambda$ is in the degree $|\lambda|$ piece of $A^\bullet(G(r,n))$ i.e.\ $\sigma_\lambda \in A^{|\lambda|}(G(r,n))$. The top piece $A^{r(n-r)}(G(r,n))$ is generated by the Schubert cycle $\sigma_{r\times(n-r)}$ and the isomorphism with $\Z$ sending $\sigma_{r\times(n-r)}$ to $1$ is called the degree map $\deg\colon A^{r(n-r)}(G(r,n)) \to \Z$.

The product in $A^\bullet(G(r,n))$ can be described purely by partition combinatorics known as Schubert calculus. There are two important formulas to do computations. First is the complimentary dimension formula which states that if $|\lambda| + |\mu| = r(n-r)$ then $\deg(\sigma_\lambda \sigma_\mu) = 1$ if $\mu = \lambda^c$ and $0$ otherwise. The other formula is called Pieri's formula. For a partition $\lambda \subseteq r\times(n-r)$ and a positive integer $b$ we have $\sigma_\lambda \sigma_{(b)} = \sum_{\mu \in \lambda \otimes b} \sigma_\mu$, where $\lambda \otimes b$ is the set of partitions in $r\times (n-r)$ that can be obtained from the Young diagram of $\lambda$ by adding $b$ boxes, at most one per column.  For a comprehensive reference see \cite{EisHar}.

\subsection{Matroids}
We mainly use the basis definition of a matroid, and identify a matroid $M$ of rank $r$ on $[n]$ with its set of bases $\mathcal{B} \subseteq \binom{[n]}{r}$. For references on matroid theory see \cite{oxleyBook}, \cite{Katz} and \cite{ardila}. We denote the number of connected components of $M$ by $\kappa(M)$.

To a point $x\in G(r,n)$ we associate a matroid $M_x$. The bases of $M_x$ are the sets $I \in \binom{[n]}{r}$ with non-zero Plücker coordinates $p_I(x) \neq 0$. A matroid is called representable over $\CC$ if it is obtained from some point in $G(r,n)$. We will be interested in a subclass of matroids called sparse paving matroids. A matroid $M$ of rank $r$ is called paving if all its circuits have cardinality $r$ or $r+1$. A matroid $M$ is called sparse paving if both $M$ and its dual $M^*$ are paving. It is conjectured in \cite{MAYHEW} and \cite{oxleyBook} that asymptotically almost all matroids are sparse paving.

We will also describe another class of matroids called lattice path matroids in order to make several other definitions easier; see \cite{BoninMierNoy} for comprehensive definitions. Let $P$ and $Q$ be north-east lattice paths from $(0,0)$ to $(n-r,r)$ such that $P$ never goes below $Q$. The lattice path matroid $M[P,Q]$ is a rank $r$ matroid on $[n]$ with bases in bijection with the lattice paths from $(0,0)$ to $(n-r,r)$ bounded by $P$ and $Q$. To get a basis $B$ of $M[P,Q]$ from a lattice path $S$, label the steps of $S$ by the integers $1,\dots, n$ from $(0,0)$ to $(n-r,r)$. Then the basis $B$ is the set of labels of vertical steps.

One special case of lattice path matroids are those where $P$ is the path starting with $r$ vertical steps followed by $n-r$ horizontal steps. In this case the matroids $M[P,Q]$ are called Schubert matroids\footnote{Schubert matroids are known under many names in the literature, such as nested matroids and generalized catalan matroids.}. If $I \in \binom{[n]}{r}$ is the set indicating the vertical steps of $Q$ we denote the Schubert matroid by $SM_I = M[P,Q]$. We introduce the matroids that will be most important to us in the following examples.
\begin{example}
    \label{ex:uniform}
    The Schubert matroid given by $I = \{n-r+1,n-r+2,\dots, n\}$ is the uniform matroid $SM_I = U_{r,n}$. To see this consider the example $r=3$ and $n = 7$. Since the paths $P$ and $Q$ form the $3$ by $4$ rectangle, the bases of $SM_I$ are in bijection with the lattice paths from $(0,0)$ to $(4,3)$. We see that any set $B \in \binom{[7]}{3}$ can be obtained as labels of vertical steps of some path.
\end{example}

\begin{example}
    \label{ex:minimal}
    The Schubert matroid given by $I = \{2,3,\dots,r-1,n \}$ is called the minimal matroid, and is denoted $T_{r,n} = SM_I$. It is, up to isomorphism, the unique connected matroid of rank $r$ on $[n]$ with minimal number of bases. In the example where $r=3$ and $n=7$ it is given by the diagram
    \begin{equation*}
        \scalebox{0.8}{\young(~~~~,~,~)}
    \end{equation*}
    and is represented by the columns of the matrix
    \begin{equation*}
        \left(\begin{array}{ccccccc}
            1 & 0 & 0 & 1 & 1 & 1 & 1\\
            0 & 1 & 0 & 1 & 1 & 1 & 1\\
            0 & 0 & 1 & 1 & 1 & 1 & 1\\
        \end{array}\right).
    \end{equation*}
    In general $T_{r,n}$ is obtained from $U_{r,r+1}$ by adding $n-r-1$ elements parallel to one of the elements in $U_{r,r+1}$.
\end{example}
\begin{example}
    \label{ex:panhandle}
    The Panhandle matroids were introduced in \cite{HMMMNVYPanhandle}. They are given by three positive integers $r \leq s < n$. Here we define them as the Schubert matroid $\Pan_{r,s,n} = SM_I$ where $I = \{s-r+2, s-r+3, \dots ,s, n\}$. The diagram of the panhandle matroid $\Pan_{4,6,10}$ is
    \begin{equation*}
        \scalebox{0.8}{\young(~~~~~~,~~~,~~~,~~~)}
    \end{equation*}
    which looks like a panhandle, hence the name. Notice that $\Pan_{r,r,n} = T_{r,n}$ and $\Pan_{r,n-1,n} = U_{r,n}$.
\end{example}

\subsection{Matroid polytopes and invariants}
In this section we follow the conventions of \cite{DerksenFink} and \cite{FinkSpeyer}. Throughout let $S(r,[n])$ denote the set of matroids of rank $r$ on $[n]$. We are interested in functions $f\colon S(r,[n]) \to A$ with values in an abelian group $A$. Such functions are called matroid invariants if $f(M_1)=f(M_2)$ whenever $M_1$ and $M_2$ are isomorphic. Also $f$ is called additive if $f(M) = 0$ whenever $M$ is not connected. Another important property of functions of matroids is called valuativity, which we will explain in what follows.

To any matroid $M$ on $[n]$ the matroid (base) polytope $P(M) \subset \R^n$ is defined as the convex hull of the indicator vectors $e_B = \sum_{i\in B}e_i$ for all bases $B$ of $M$. In \cite{Gel} it was shown that an equivalent definition of matroids can be given in terms of its matroid polytope, so the matroid polytope $P(M)$ contains all the information of $M$. The dimension of $P(M)$ is $n - \kappa(M)$. We will be interested in how we can chop up matroid polytopes into smaller matroid polytopes.
\begin{definition}
    A matroid subdivision of a matroid $M$ is a polyhedral complex consisting of matroid polytopes with facets $P_1, P_2, \dots P_k$ such that 
    \begin{equation*}
        P(M) = P_1 \cup P_2 \cup \dots \cup P_k.
    \end{equation*}
\end{definition}

Given such a subdivision, for any set $J \subseteq [k]$ we denote $P_J = \bigcap_{j\in J}P_j$ and $P_\emptyset = P(M)$.
We may identify the set $S(r,[n])$ with the set of matroid polytopes of matroids in $S(r,[n])$.

\begin{definition}
    A function $f\colon S(r,[n]) \to A$ is called valuative if for any $M \in S(r,[n])$ and any matroid subdivision of $M$ as above, the relation
    \begin{equation*}
        \sum_{J\subseteq [k]} (-1)^{|J|}f(P_J) = 0
    \end{equation*}
    holds.
\end{definition}
That is, $f$ is valuative if it satisfies an inclusion-exclusion property with respect to matroid subdivisions. Notice that if $f$ is both valuative and additive, since the matroid of $P_J$ is disconnected for all $J$ except $1,2, \dots , k$ and $\emptyset$, the formula simplifies to
\begin{equation*}
    f(M) = \sum_{i=1}^k f(P_i).
\end{equation*}
 In \cite{DerksenFink}, Derksen and Fink proved that a function $f\colon S(r,[n]) \to A$ is valuative if and only if it factors through the abelian group generated by indicator functions of matroid polytopes of $S(r,[n])$, denoted $P(r,[n])$. They also show that the quotient of $P(r,[n])$ given by identifying isomorphic matroids is generated as an abelian group by the classes of the indicator functions of Schubert matroids. This means that any valuative matroid invariant is determined by its values on the Schubert matroids. 

 The most obvious example of a valuative matroid invariant is the volume of a matroid polytope.
 \begin{example}
      Let $M$ be a matroid on $[n]$. We denote the normalized $s$ dimensional volume of $P(M)$ by $\Vol_s(P(M))$. Then
     \begin{equation*}
         \Vol_s(P(-))\colon S(r,[n]) \to \Z
     \end{equation*}
     is a valuative matroid invariant. Moreover, if $s = n-1$ it is an additive matroid invariant since $M$ is connected if and only if $n-\kappa(M) = s$. We may drop the subscript in $\Vol_s(P(M))$ if it does not cause too much confusion.
\end{example}
The next important example was introduced by Crapo in \cite{Crapo}.
\begin{example}
    The beta invariant is a function 
    \begin{equation*}
        \beta \colon S(r,[n]) \to \Z.
    \end{equation*}
    It is completely determined by the properties that if $i$ is neither a loop or coloop of $M$ then
    \begin{equation*}
        \beta(M) = \beta(M/i) + \beta(M\minus i).
    \end{equation*}
    And the base cases: \(\beta(U_{0,1}) = 0\), \(\beta(U_{1,1}) = 1\), and otherwise if $M$ contains a loop or coloop then $\beta(M) = 0$.
    It is well known that $\beta(M)$ is non-negative for all $M$ and $\beta(M) = 0$ if and only if $M$ is disconnected (excluding the special case $U_{0,1}$). The beta invariant is in fact also a valuative and additive matroid invariant.
\end{example}

The next example calculates the beta invariant of minimal matroids.
\begin{example}
    As we have seen, the minimal matroid $T_{r,n}$ is obtained from $U_{r,r+1}$ by adding $n-r-1$ elements parallel to one of the elements in $U_{r,r+1}$. Let $i$ be one of these parallel elements. Then $T_{r,n} \minus i = T_{r,n-1}$ and $T_{r,n}/i$ contains $n-r-1$ loops, so
    \begin{equation*}
        \beta(T_{r,n}) = \beta(T_{r,n}/i) + \beta(T_{r,n} \minus i) = \beta(T_{r,n-1}).
    \end{equation*}
    We may iterate this recursion until $T_{r,m}$ does not contain any parallel elements to choose, i.e.\ when $m-r-1 = 0$. Hence
    \begin{equation*}
        \beta(T_{r,n}) = \beta(T_{r,r+1}) = \beta(U_{r,r+1}).
    \end{equation*}
    The beta invariant of a uniform matroid is known to be
    \begin{equation*}
        \beta(U_{r,n}) = \binom{n-2}{r-1}
    \end{equation*}
    so
    \begin{equation*}
        \beta(T_{r,n}) = \beta(U_{r,r+1}) = \binom{r-1}{r-1} = 1.
    \end{equation*}
\end{example}

The following theorem characterises sparse paving matroids in terms of matroid subdivisions, and can be found for example in \cite{FerroniMinimal} or in \cite{JoswigSchröter}.
\begin{thm}
    A matroid $M \in S(r,[n])$ is sparse paving if and only if there exists a matroid subdivision of $U_{r,n}$ with facets $P(M)$ and $P(M_1), \dots , P(M_k)$ where all matroids $M_i$ are isomorphic to $T_{r,n}$.
\end{thm}
Hence for any additive and valuative matroid invariant $f$ and sparse paving matroid $M$ of rank $r$ on $[n]$ we have
\begin{align*}
    f(U_{r,n}) &= f(M) + k\cdot f(T_{r,n})\\
    f(M) &= f(U_{r,n}) - k\cdot f(T_{r,n}).
\end{align*}


%% file: Section3_Schubert_Coefficients.tex
\section{Schubert coefficients}

In this section we define Schubert coefficients of matroids and recap what is previously known. For each point $x\in G(r,n)$ we may consider the class of the torus orbit closure $\T$ in $A^\bullet(G(r,n))$. We express this class in the basis of Schubert cycles as
\begin{equation}
    \label{eq:TClass}
    [\T] = \sum_{\lambda} d_\lambda \sigma_\lambda,
\end{equation}
where the sum is over all partitions $\lambda \subseteq r\times(n-r)$ with $|\lambda| = \codim(\T)$ and $d_\lambda \in \Z$. More precisely the coefficient $d_\lambda = \deg([\T] \sigma_{\lambda^c})$ is the number of points in the intersection of $\T$ and the Schubert variety $X_{\lambda^c}(V_\bullet)$ for a generic flag $V_\bullet$. Hence the coefficients are non-negative.

In \cite{Gel} it is shown that $\T$ is isomorphic to the projective toric variety of the matroid polytope $P(M_x)$. Hence many properties of $\T$ can be read off from the matroid $M_x$. For example
\begin{equation}
    \dim(\T) = \dim(P(M_x)) = n - \kappa(M_x)
\end{equation}
and
\begin{equation}
    \deg(\T) = \Vol(P(M_x)).
\end{equation}
Here $\deg(\T)$ denotes the degree of $\T$ embedded in $\PP^{\binom{n}{r}-1}$ and $\Vol(P(M_x))$ is the normalized $n-\kappa(M_x)$ dimensional volume of $P(M_x)$.

As we will see, the class of $\T$ in $A^\bullet(G(r,n))$ only depends on the matroid $M_x$, so we may denote the coefficients in \cref{eq:TClass} by $d_\lambda(M_x)$. The coefficient $d_\lambda(M_x)$ only makes sense if $|\lambda| = r(n-r)-(n-\kappa(M_x))$. For all other cases we define $d_\lambda(M_x) = 0$. This way we can consider $d_\lambda$ as functions from the set of matroids representable over $\CC$ of rank $r$ on $[n]$ to $\Z$. We call the integers $d_\lambda(M_x)$ the Schubert coefficients of the matroid $M_x$.

In \cite{Klyachko}, Klyachko gives a formula for $d_\lambda(U_{r,n})$ in terms of evaluations of Schur polynomials.

\begin{thm}\textnormal{\cite[Theorem 6]{Klyachko}}
    \label{thm:Klyachko}
    Let $\lambda \subseteq r \times (n-r)$ be a partition with $|\lambda| = (r-1)(n-r-1)$. The associated Schubert coefficient of $U_{r,n}$ is
    \begin{equation*}
        d_\lambda(U_{r,n}) = \sum_{i=0}^r (-1)^i\binom{n}{i}s_{\lambda^c}(1^{r-i}).
        \qedhere
    \end{equation*}
\end{thm}

In \cite{Speyer}, Speyer proved the following important theorem describing the Schubert coefficient associated to the hook complement.
\begin{thm}\textnormal{\cite[Theorem 5.1]{Speyer}}
    \label{thm:Speyer}
    For $x \in G(r,n)$
    \begin{equation*}
        d_{h^c}(M_x) = \beta(M_x).
    \end{equation*}
\end{thm}
Notice that this is true even if $M_x$ is disconnected, since in that case both $d_{h^c}(M_x)$ and $\beta(M_x)$ are zero.

One issue with the Schubert coefficients of matroids is that points in the Grassmannian only give rise to matroids representable over $\CC$. This is rectified by Fink and Speyer in \cite{FinkSpeyer} where they define the class of any matroid \(M\) in the \(K\)-theory of the Grassmannian, which can be specialized to a class in the Chow ring. We denote this class \(\Sc(M) \in A^\bullet(G(r,n))\). If \(M\) is represented by a point \(x \in G(r,n)\), then \(\Sc(M) = [\T]\). The specialization from \(K\)-theory to the Chow ring is made explicit by Berget and Fink in \cite{BergetFink}. Note that $\Sc(M) \in A^{r(n-r)-(n-\kappa(M))}(G(r,n))$. From the class $\Sc(M)$ we get Schubert coefficients $d_\lambda(M)$ for arbitrary matroids. We summarize the most relevant takeaways from \cite{FinkSpeyer} and \cite{BergetFink} in a theorem.
\begin{thm}
\label{thm:BergetFink}
    For each partition $\lambda \subseteq r\times (n-r)$ there are functions $d_\lambda \colon S(r,[n]) \to \Z$ coinciding with the coefficients of \cref{eq:TClass} whenever $M = M_x$ for $x\in G(r,n)$. These functions are valuative matroid invariants and whenever $|\lambda| = r(n-r)-(n-1)$ then $d_\lambda$ is additive.
\end{thm}
Notice that since valuative matroid invariants are determined by their values on Schubert matroids, which are representable, it follows from \cref{thm:Speyer} and \cref{thm:BergetFink} that \(d_{h^c}(M) = \beta(M)\) for all matroids \(M\).

As mentioned in \cite{BergetFink} there is another way of defining Schubert coefficients of arbitrary matroids by the theory of tautological classes of matroids from \cite{BEST}. In this setting $d_\lambda(M) = \deg_{X_n}(s_{\lambda^c}(\mathcal{S}_M^\vee))$ where $X_n$ is the $(n-1)$-dimensional permutahedral variety, $\mathcal{S}_M$ is the tautological sub-bundle of $M$ and $s_{\lambda^c}(\mathcal{S}_M^\vee)$ is the evaluation of the Schur polynomial $s_{\lambda^c}$ at the Chern roots of $\mathcal{S}_M^\vee$.

For representable matroids $M_x$ we could conclude that $d_\lambda(M_x)$ is non-negative since it is the number of points in an intersection of varieties. This argument only works for representable matroids, but the non-negativity of $d_\lambda(M)$ for arbitrary matroids $M$ is conjectured in \cite{BergetFink}.
\begin{conj}\textnormal{\cite[Conjecture 9.13]{BergetFink}}
    \label{conj:positivity}
    The Schubert coefficients $d_\lambda(M)$ are non-negative for all matroids $M$.
\end{conj}

%% file: Section4_disconnected.tex
\section{Schubert coefficients of disconnected matroids}

In this section we prove the following result, expressing the Schubert coefficients of a disconnected matroid in terms of the Schubert coefficients of its connected components.
\begin{thm}
    \label{thm:disconnected}
    Let $M_i$ be a matroid of rank $r_i$ on $[n_i]$ for $i = 1, 2$. The class of $M_1\oplus M_2$ in $A^\bullet(G(r_1+r_2,n_1+n_2))$ is
    \begin{equation*}
        \Sc(M_1\oplus M_2) = (\square \Sc(M_1))(\square \Sc(M_2))
    \end{equation*}
    where
    \begin{equation*}
        \square \Sc(M_i) = \sum_\mu d_\mu(M_i) \sigma_{\square \mu},
    \end{equation*}
    and \(\square \mu\) is obtained from \(\mu\) by attaching a \(r_i\) by \(n_{3-i} - r_{3-i}\) rectangle to the left of the Young diagram of \(\mu\).
\end{thm}
Notice that there is no assumption on \(M_i\) being connected in \cref{thm:disconnected}, so if \(M\) has more than \(2\) connected components applying the theorem repeatedly yields a formula for \(\Sc(M)\) in terms of the Schubert coefficients of its connected components.

\begin{example}
    We use \cref{thm:disconnected} to compute \(\Sc(U_{2,4} \oplus U_{2,5})\). By \cref{thm:Klyachko} we have
    \begin{align*}
        \Sc(U_{2,4}) &= 2 \sigma_{(1,0)}, \\
        \Sc(U_{2,5}) &= 3 \sigma_{(2,0)} + 1 \sigma_{(1,1)}.
    \end{align*}
    Attaching a \(2\times 3\) rectangle to the left of \((1,0)\) and a \(2 \times 2\) rectangle to the left of \((2,0)\) and \((1,1)\) gives
    \begin{align*}
        \square\Sc(U_{2,4}) &= 2 \sigma_{(4,3)}, \\
        \square \Sc(U_{2,5}) &= 3 \sigma_{(4,2)} + 1 \sigma_{(3,3)}.
    \end{align*}
    Now computing the product \((\square\Sc(U_{2,4}))(\square \Sc(U_{2,5}))\) in \(A^\bullet(G(4,9))\) gives
    \begin{align*}
        \Sc(U_{2,4} \oplus U_{2,5}) &= (2 \sigma_{(4,3)})(3 \sigma_{(4,2)} + 1 \sigma_{(3,3)})\\
        &= 2\sigma_{(4, 3, 3, 3)} + 8\sigma_{(4, 4, 3, 2)} + 6\sigma_{(4, 4, 4, 1)} + 8\sigma_{(5, 3, 3, 2)} \\
        & + 8\sigma_{(5, 4, 2, 2)} + 14\sigma_{(5, 4, 3, 1)} + 6\sigma_{(5, 4, 4, 0)} + 8\sigma_{(5, 5, 2, 1)} + 6\sigma_{(5, 5, 3, 0)}.
    \end{align*}
\end{example}

Throughout fix $r_1 < n_1$ and $r_2 < n_2$ and let $r = r_1 + r_2$ and $n = n_1 + n_2$. We consider the embedding 
\begin{equation*}
    \iota \colon G(r_1,n_1) \times G(r_2,n_2) \hookrightarrow G(r,n)
\end{equation*}
defined by $L(\iota(x,y)) = L(x) \oplus L(y) \subseteq \CC^{n_1} \oplus \CC^{n_2} = \CC^n$.

\begin{lemma}
    \label{lem:varietyLift}
    Let $X$ be a subvariety of $G(r_1,n_2)$ of codimension $s$ with
    \begin{equation*}
        [X] = \sum_\mu d_\mu(X) \sigma_\mu \in A^\bullet(G(r_1,n_1))
    \end{equation*}
    for $d_\mu(X) \in \ZZ$ and let $\tilde{X} = \iota(X \times G(r_2,n_2))$. Then
    \begin{equation*}
        [\tilde{X}] = \sum_\mu d_\mu(X) \sigma_{\square \mu} \in A^\bullet(G(r,n))
    \end{equation*}
    where both sums are over all partitions $\mu \subseteq r_1 \times (n_1-r_1)$ with $|\mu| = s$ and $\square \mu = (n_2-r_2 + \mu_1, \dots, n_2-r_2 + \mu_{k_1}) \subseteq r\times (n-r)$.
\end{lemma}
\begin{proof}
    For fixed $\mu$ we need to show that 
    \[
        \deg([\tilde{X}] \sigma_{(\square \mu)^c}) = \deg([X] \sigma_{\mu^c}) = d_\mu(X).
    \]
    Let $V_\bullet$ be a complete flag of $\CC^{n_1}$ such that the intersection $X \cap X_{\mu^c}(V_\bullet)$ is transverse. Fix a $k_2$-dimensional subspace $U$ of $\CC^{n_2}$ and let $V^+_\bullet$ be a complete flag of $\CC^n = \CC^{n_1} \oplus \CC^{n_2}$ such that $V^+_{k_2} = 0 \oplus U$ and  $V^+_{n_2 + j} = V_j \oplus \CC^{n_2}$ for $j = 1, 2, \dots, n_1$. Let $(j_i)_{i=1}^{r_1}$ be the jumping sequence of $\mu^c$ in $r_1 \times (n_1-r_1)$. Then notice that the complement of $\square \mu$ in $r \times (n-r)$ is the partition obtained by attaching the $r_2 \times (n-r)$ rectangle to the top of the Young diagram of $\mu^c$ i.e. the jumping sequence of $(\square \mu)^c$ is
    \begin{equation*}
        (1,2,\dots, r_2, n_2 + j_1, n_2 + j_2, \dots n_2 + j_{r_1}).
    \end{equation*}
    From the first part of the jumping sequence above only the last term contributes with a defining inequality of the corresponding Schubert variety. So $X_{(\square \mu)^c}(V^+_\bullet)$ is given by 
    \begin{equation*}
        \left\{z \in G(r,n) \mid \dim(L(z) \cap V^+_{r_2}) \geq r_2, \ \dim(L(z) \cap V^+_{n_2+j_i}) \geq r_2 + i\right\}.
    \end{equation*}
    We need to show that $d_\mu(X) = |X \cap X_{\mu^c}(V_\bullet)| = |\tilde{X} \cap X_{(\square \mu)^c}(V^+_\bullet)|$. Suppose $z \in \tilde{X} \cap X_{(\square \mu)^c}(V^+_\bullet)$, since $z \in \tilde{X}$,  $L(z) = L(x) \oplus L(y)$ for $x \in X$ and $y \in G(r_2,n_2)$. Now
    \begin{equation*}
        L(z) \cap V^+_{r_2} = (L(x) \oplus L(y)) \cap (0 \oplus U) = 0 \oplus (L(y) \cap U)
    \end{equation*}
    so the first defining inequality of $X_{(\square \mu)^c}(V^+_\bullet)$ becomes 
    \begin{equation*}
        \dim(0\oplus (L(y) \cap U)) \geq r_2 \iff L(y) = U
    \end{equation*}
    since both $L(y)$ and $U$ are of dimension $r_2$. For the second type of defining inequality of $X_{(\square \mu)^c}(V^+_\bullet)$ we have
    \begin{equation*}
        L(z) \cap V^+_{n_2 + j_i} = (L(x) \oplus L(y)) \cap (V_{j_i} \oplus \CC^{n_2}) = (L(x) \cap V_{j_i}) \oplus U.
    \end{equation*}
    so 
    \begin{equation*}
        \dim(L(z) \cap V^+_{n_2 + j_i}) \geq k_2 + i \iff \dim(L(x) \cap V_{j_i}) \geq i
    \end{equation*}
    which are the defining inequalities of $X_{\mu^c}(V_\bullet)$. Hence if $z = \iota(x,y) \in \tilde{X} \cap X_{(\square \mu)^c}(V^+_\bullet)$ then $x \in X \cap X_{\mu^c}(V_\bullet)$ and if $x \in X \cap X_{\mu^c}(V_\bullet)$ then $\iota(x,U) \in \tilde{X} \cap X_{(\square \mu)^c}(V^+_\bullet)$ so $|X \cap X_{\mu^c}(V_\bullet)| = |\tilde{X} \cap X_{(\square \mu)^c}(V^+_\bullet)|$.
\end{proof}

By a similar argument for subvarieties $Y$ of $G(r_2,n_2)$ with 
\begin{equation*}
    [Y] = \sum_\eta d_\eta(Y) \sigma_\eta \in A^\bullet(G(r_2,n_2)),
\end{equation*}
let $\tilde{Y} = \iota(G(r_1,n_2) \times Y)$. Then we have
\begin{equation*}
    [\tilde{Y}] = \sum_\eta d_\eta(Y) \sigma_{\square \eta} \in A^\bullet (G(r,n))
\end{equation*}
where $\square \eta = (n_1-r_1+\eta_1, \dots, n_1-r_1 + \eta_{r_2}) \subseteq r \times (n-r)$.

Using \cref{lem:varietyLift} we can now prove \cref{thm:disconnected} in the case where $M_1$ and $M_2$ are representable over $\CC$.
\begin{proof}
    Fix $x \in G(r_1,n_1)$ and $y \in G(r_2,n_2)$ and let $z = \iota(x,y)$. Let $T_i = (\CC^*)^{n_i}$ be the algebraic torus acting on $G(r_i,n_i)$  for $i = 1,2$ and $T = T_1 \times T_2 = (\CC^*)^n$ be the algebraic torus acting on $G(r,n)$. Notice that
    \begin{equation*}
        \overline{Tz} = \iota(\overline{T_1 x} \times \overline{T_2 y}) = \iota(\overline{T_1 x} \times G(r_2,n_2)) \cap \iota(G(r_1,n_2) \times \overline{T_2 y}),
    \end{equation*}
    hence $[\overline{Tz}] = [\tilde{\overline{T_1 x}}][\tilde{\overline{T_2 y}}]$.
    
    Since $M_z = M_x \oplus M_y$ applying \cref{lem:varietyLift} to $\overline{T_1 x}$ and $\overline{T_2 y}$ gives
    \begin{equation*}
        \Sc(M_z) = (\sum_\mu d_\mu(M_x) \sigma_{\square \mu}) (\sum_\eta d_\eta(M_y) \sigma_{\square \eta}) = (\square Sc(M_x))(\square Sc(M_y)). \qedhere
    \end{equation*}
\end{proof}

To prove \cref{thm:disconnected} for arbitrary matroids, we use the fact that valuative matroid invariants are determined by their values on Schubert matroids, which are representable over $\CC$. In particular if two valuative matroid invariants are equal on all matroids representable over $\CC$, they are equal on all matroids.

\begin{proof}[proof of \cref{thm:disconnected}]
    Fix a matroid $M_1 \in S(r_1,n_1)$ that is representable over $\CC$. Notice that both
    \begin{equation*}
        \Sc(M_1 \oplus - ) \colon S(r_2,n_2) \to A^\bullet(G(r,n))
    \end{equation*}
    and 
    \begin{equation*}
        (\square Sc(M_1)) (\square Sc(-)) \colon S(r_2,n_2) \to A^\bullet(G(r,n))
    \end{equation*}
    are valuative matroid invariants that are equal on matroids representable over $\CC$ hence $\Sc(M_1 \oplus M_2) = (\square \Sc(M_1)) (\square \Sc(M_2))$ for all matroids $M_2 \in S(r_2,n_2)$.
    
    Now fix a matroid $M_2 \in S(r_2,n_2)$ and notice again that 
    \begin{equation*}
        \Sc(- \oplus M_2) \colon S(r_1,n_1) \to A^\bullet(G(r,n))
    \end{equation*}
    and 
    \begin{equation*}
        (\square \Sc(-)) (\square \Sc(M_2)) \colon S(r_1,n_1) \to A^\bullet(G(r,n))
    \end{equation*}
    are valuative matroid invariants that are equal on matroids representable over $\CC$ hence $\Sc(M_1 \oplus M_2) = (\square \Sc(M_1)) (\square \Sc(M_2))$ for all matroids $M_1 \in S(r_1,n_1)$ and $M_2 \in S(r_2,n_2)$.
\end{proof}

\begin{corollary}
\label{cor:disconnected}
    Fix a partition $\lambda \subset r \times (n-r)$ with $|\lambda| = r(n-r) - (n - \kappa(M_1 \oplus M_2))$. Then the Schubert coefficient $d_\lambda(M_1 \oplus M_2)$ is given by
    \begin{equation*}
        d_\lambda(M_1 \oplus M_2) = \sum_{\mu, \eta} d_\mu(M_1) d_\eta(M_2) c_{\square\mu \square\eta}^\lambda
    \end{equation*}
    where $c_{\square\mu \square\eta}^\lambda$ is a Littlewood-Richardson coefficient.
\end{corollary}
\begin{proof}
    By \cref{thm:disconnected} we have
    \begin{align*}
        \Sc(M_1 \oplus M_2) &= (\sum_\mu d_\mu(M_1) \sigma_{\square \mu}) (\sum_\eta d_\eta(M_2) \sigma_{\square \eta})\\
        &= \sum_{\mu, \eta} d_\mu(M_1) d_\eta(M_2) \sigma_{\square \mu} \sigma_{\square \eta}
    \end{align*}
    and since $d_\lambda(M_1\oplus M_2) = \deg(\Sc(M_1 \oplus M_2) \sigma_{\lambda^c})$ we have
    \begin{align*}
        d_\lambda(M_1 \oplus M_2) &= \sum_{\mu, \eta} d_\mu(M_1) d_\eta(M_2) \deg(\sigma_{\square \mu} \sigma_{\square \eta} \sigma_{\lambda^c})\\
        &= \sum_{\mu, \eta} d_\mu(M_1) d_\eta(M_2) c_{\square\mu \square\eta}^\lambda
    \end{align*}
    All the sums above are over partitions $\mu \subseteq r_1 \times (n_1 - r_1)$ with $|\mu| = r_1(n_1-r_1) - (n_1 - \kappa(M_1))$ and $\eta \subseteq r_2 \times (n_2 - r_2)$ with $|\eta| = r_2(n_2-r_2) - (n_2 - \kappa(M_2))$.
\end{proof}

Notice that the formula for \(d_\lambda(M_1\oplus M_2)\) in \cref{cor:disconnected} preserves non-negativity of the Schubert coefficients. That is if both \(M_1\) and \(M_2\) have non-negative Schubert coefficients, so does \(M_1 \oplus M_2\). So it is enough to prove \cref{conj:positivity} for connected matroids.

%% file: Section5_Sparse_Paving.tex
\section{Schubert coefficients of sparse paving matroids}
In this section we compute the Schubert coefficients of sparse paving matroids. The observation that allows us to do so is the following.
\begin{proposition}
    \label{prop:VolDeg}
    For $x \in G(r,n)$ let $s = \dim(\T) = n - \kappa(M_x)$. The following equality holds
    \begin{equation}
        \label{eq:impLem}
        \deg([\T]\sigma_{(1)}^s) = \Vol_s(P(M_x)).
    \end{equation}
    This gives the linear relation among the Schubert coefficients of $M_x$ 
    \begin{equation}
        \label{eq:linRel}
        \sum_\lambda c_{\lambda^c} d_\lambda(M_x) = \Vol_s(P(M_x)),
    \end{equation}
    where $c_{\lambda^c}$ are the integer coefficients found in the expansion of $\sigma_{(1)}^s$. That is,  $c_{\lambda^c} = \deg(\sigma_{(1)}^s \sigma_{\lambda}).$ Here the sum is taken over all partitions $\lambda \subseteq r\times(n-r)$ with $|\lambda| = r(n-r) - s$.
\end{proposition}
\begin{proof}
    There are two ways of computing the degree $D$ of $\T$ in $\PP^{\binom{n}{r}-1}$. First, since $\T$ is a subvariety of $G(r,n)$, $D = \deg([\T]\sigma_{(1)}^s)$. Second, since $\T$ is isomorphic to the projective toric variety obtained from $P(M_x)$, $D = Vol_s(P(M_x))$. Expanding $[\T]$ and $\sigma_{(1)}^s$ gives
    \begin{align*}
        [\T] &= \sum_\lambda d_\lambda(M_x)\sigma_\lambda \text{ and }\\
        \sigma_{(1)}^s &= \sum_\lambda c_{\lambda^c} \sigma_{\lambda^c}.
    \end{align*}
    Now by the complimentary dimension formula the left hand side of \cref{eq:impLem} becomes the left hand side of \cref{eq:linRel}.
\end{proof}

The proof above only works for matroids representable over $\CC$, but we can extend the result to arbitrary matroids using the fact that valuative matroid invariants are determined by their values on Schubert matroids.
\begin{corollary}
    Let $M$ be a rank $r$ matroid on $[n]$ and $s = n - \kappa(M)$ then 
    \begin{equation*}
        \deg(\Sc(M)\sigma_{(1)}^s) = Vol_s(P(M)).
    \end{equation*}
    This gives the linear relation 
    \begin{equation*}
        \sum_\lambda c_{\lambda^c} d_\lambda(M) = \Vol_s(P(M)),
    \end{equation*}
    where $c_{\lambda^c}$ are as in \cref{prop:VolDeg}.
\end{corollary}
\begin{proof}
    For a fixed positive integer $s \leq r(n-r)$, notice that both 
    \begin{align*}
        \Vol_s(P(-)) \colon & S(r,[n]) \to \ZZ \text{ and }\\
        \deg \big((\sum_\lambda d_\lambda(-) \sigma_\lambda )\sigma_{(1)}^s \big) \colon & S(r,[n]) \to \ZZ,
    \end{align*}
    where the sum is over all partitions $\lambda \subseteq r\times(n-r)$ with $|\lambda| = r(n-r)-s$, are valuative matroid invariants. By Theorem 6.3 of \cite{DerksenFink} all valuative matroid invariants are determined by their values on the Schubert matroids $SM_I$. Since the Schubert matroids are representable over $\CC$ we know that $\deg((\sum_\lambda d_\lambda(SM_I) \sigma_\lambda )\sigma_{(1)}^s) = \Vol_s(P(SM_I))$ by \cref{prop:VolDeg}. Hence these invariants are equal for any matroid. The corollary follows from setting $s = n-\kappa(M)$.
\end{proof}

We will use this to compute the Schubert coefficients of sparse paving matroids. In light of the previous section we may assume that all matroids are connected. First we compute The Schubert coefficients of the minimal matroids.
\begin{lemma}
    \label{lem:minimalMatroid}
    The Schubert coefficients of the minimal matroid $T_{r,n}$ are
    $d_\lambda(T_{r,n}) = 0$ for $\lambda \neq h^c$ and $d_{h^c}(T_{r,n}) = \beta(T_{r,n}) = 1$, i.e.\ $\Sc(T_{r,n}) = \sigma_{h^c}$.
\end{lemma}

\begin{proof}
    By \cref{prop:VolDeg}
    \begin{equation}
        \label{eq:minimalMatroidEq}
        \sum_\lambda c_{\lambda^c} d_\lambda(T_{r,n}) = \Vol(P(T_{r,n})).
    \end{equation}
    The matroid polytope $P(T_{r,n})$ is a pyramid of lattice height $1$ over the product of simplices $\Delta_{r-1} \times \Delta_{n-r-1}$ and hence has volume $\Vol(P(T_{r,n})) = \binom{n-2}{r-1}$; see also example 7.4 of \cite{FerSch}. Next we compute the coefficient of the hook partition $c_h$. By using Pieri's formula and considering the expansion of $\sigma_{(1)}^{n-1}$ in $A^\bullet(G(r,n))$ we see that $c_{\lambda^c} = |SYT(\lambda^c)|$ is the number of standard Young tableaux with shape $\lambda^c$. The following formula for $|SYT(\lambda^c)|$ is given in Theorem $1$ of \cite{FrameRobinsonThrall}\footnote{This paper gives the formula for the degree of the corresponding irreducible representation of the symmetric group on $|\lambda|$ elements. There is a well known basis of this representation in bijection with standard Young tableaux of shape $\lambda$.}:
    \begin{equation*}
        |SYT(\lambda^c)| = \frac{|\lambda^c|!}{\prod_{b\in \lambda^c} hl(b)}
    \end{equation*}
    where the product is over all boxes in the Young diagram of $\lambda^c$ and $hl(b)$ is the hook length of $b$. Applying this formula to the hook partition $h$ gives
    \begin{align*}
        c_h = |SYT(h)| &= \frac{(n-1)!}{(n-1)(r-1)!(n-r-1)!}\\
        &= \frac{(n-2)!}{(r-1)!(n-2 -(r-1))!}\\
        &= \binom{n-2}{r-1}
    \end{align*}
    hence $\Vol(P(T_{k,n})) = c_h$. Since $d_{h^c}(T_{r,n}) = \beta(T_{r,n}) = 1$ \cref{eq:minimalMatroidEq} becomes
    \begin{equation*}
        c_h + \sum_{\lambda\neq h^c} c_{\lambda^c}d_\lambda(T_{r,n}) = \Vol(P(T_{r,n})).
    \end{equation*}
    The integers $c_{\lambda^c}$ are positive since it is the degree of the Schubert variety $X_\lambda(V_\bullet)$ and the Schubert coefficients $d_\lambda(T_{r,n})$ are non-negative since $T_{r,n}$ is representable over $\CC$. Hence $d_\lambda(T_{r,n}) = 0$ for all partitions $\lambda \neq h^c$.
\end{proof}

\begin{remark}
\label{rmk:minimal}
    In light of \cref{lem:minimalMatroid} we might guess that if \(x \in G(r,n)\) represents the minimal matroid \(T_{r,n}\), then the torus orbit closure \(\T\) is equal to the Schubert variety \(X_{h^c}(V_\bullet)\). Indeed one can see that this holds if the $r$-dimensional part of \(V_\bullet\) is the vector space spanned by the standard vectors denoting the non-parallel columns of a matrix representing \(x\).
\end{remark}

By using \cref{lem:minimalMatroid} we are now ready to prove \cref{thm:1}.
\begin{proof}[Proof of \cref{thm:1}]
    Since $M$ is sparse paving there is a matroid subdivision of $P(U_{r,n})$ containing as facets $P(M)$ and $k$ polytopes of matroids isomorphic to $T_{r,n}$. Since $d_\lambda$ are additive and valuative matroid invariants this subdivision yields
    \begin{equation*}
        d_\lambda(U_{r,n}) = d_\lambda(M) + k\cdot d_\lambda(T_{r,n}) \iff d_\lambda(M)  = d_\lambda(U_{k,n}) - k\cdot d_\lambda(T_{r,n}).
    \end{equation*}
    By \cref{lem:minimalMatroid} if $\lambda \neq h^c$ then $d_\lambda(M) = d_\lambda(U_{r,n})$ and $d_{h^c}(M) = \beta(M) = \beta(U_{r,n}) - k$.
\end{proof}
Now \cref{cor:1} follows since $U_{r,n}$ is representable over $\CC$ so $d_\lambda(M) = d_\lambda(U_{r,n})$ and $\beta(M)$ are non-negative.

The proof of \cref{thm:1} applies to the following more general case.
\begin{corollary}
    Let $M$ and $M'$ be connected matroids of rank $r$ on $[n]$ and assume there is a matroid subdivision of $P(M')$ containing as facets $P(M)$ and $k$ polytopes of matroids isomorphic to $T_{r,n}$. Then $d_\lambda(M) = d_\lambda(M')$ for all $\lambda \neq h^c$.
\end{corollary}
\begin{example}
    Some well known examples of matroids that are not representable over $\CC$ are the Fano matroid $F$, the non-Pappus matroid $NP$ and the Vamos matroid $V$. These are sparse paving matroids and the classes of $F$, $NP$ and $V$ in $A^\bullet(G(r,n))$ are shown below.
    \begin{align*}
        \Sc(F) = & 6 \sigma_{(4,2,0)} + 3 \sigma_{(4,1,1)} + 3 \sigma_{(3,3,0)} + 8 \sigma_{(3,2,1)} + \sigma_{(2,2,2)} \in A^\bullet(G(3,7))\\
        \Sc(NP) = & 15 \sigma_{(6,4,0)} + 15 \sigma_{(6,3,1)} + 6 \sigma_{(6,2,2)} + 13 \sigma_{(5,5,0)} \\
        & + 24 \sigma_{(5,4,1)} + 15 \sigma_{(5,3,2)}. + 6 \sigma_{(4,4,2)} + 3 \sigma_{(4,3,3)} \in A^\bullet(G(3,9))\\
        \Sc(V) = & 4 \sigma_{(4,4,1,0)} + 20 \sigma_{(4,3,2,0)} + 12 \sigma_{(4,3,1,1)} + 12 \sigma_{(4,2,2,1)} \\
        & + 15 \sigma_{(3,3,3,0)} + 20 \sigma_{(3,3,2,1)} + 4 \sigma_{(3,2,2,2)} \in A^\bullet(G(4,8))
    \end{align*}
\end{example}

One might try to expand this argument to prove non-negativity of the Schubert coefficients for a class of matroids larger than the sparse paving matroids. In what follows we give an example showing this does not work for paving matroids that are not sparse paving.

As we have seen, sparse paving matroids can be described as those matroids obtained from uniform matroids by cutting of minimal matroid polytopes from the hypersimplex. In a similar way, as described in \cite{HMMMNVYPanhandle}, one can describe paving matroids as those obtained from $U_{r,n}$ by cutting of panhandle matroids. Recall that the panhandle matroid $\Pan_{r,r,n}$ is the minimal matroid $T_{r,n}$ and $\Pan_{r,n-1,n}$ is the uniform matroid $U_{r,n}$. The next example shows that \cref{thm:1} does not hold for paving matroids that are not sparse paving.

\begin{example}
    The smallest non-minimal, non-uniform example of a panhandle matroid is $\Pan_{2,3,5}$. There is a matroid subdivision of $U_{2,5}$ of the form
    \begin{equation}
        \label{eq:PanMinUni}
        P(U_{2,5}) = P(\Pan_{2,3,5}) \cup P(T_{2,5}).
    \end{equation}
    That is $\Pan_{2,3,5}$ is a sparse paving matroid, and $T_{2,5}$ is a paving matroid that is not sparse paving. 
    By \cref{lem:minimalMatroid} and \cref{thm:Klyachko} we have
    \begin{equation*}
        d_{(1,1)}(T_{2,5}) = 0 \neq 1 = d_{(1,1)}(U_{2,5})
    \end{equation*}
    so \cref{thm:1} does not hold for paving matroids in general. For completeness by \cref{thm:Klyachko}, \cref{lem:minimalMatroid} and \cref{eq:PanMinUni} we have
    \begin{align*}
        \Sc(U_{2,5}) &= \sigma_{(1,1)} + 3 \sigma_{(2)}, \\
        \Sc(T_{2,5}) &= \hspace{1.5cm} \sigma_{(2)} \text{ and }\\
        \Sc(\Pan_{2,3,5}) &= \sigma_{(1,1)} + 2\sigma_{(2)}.
    \end{align*}
\end{example}
This leads us to state the following conjecture.
\begin{conj}
    Let $M$ be a connected matroid of rank $r$ on $[n]$ then $M$ is sparse paving if and only if $d_\lambda(M) = d_\lambda(U_{r,n})$ for all partitions $\lambda \neq h^c$.
\end{conj}

%% file: Section6_Acknowledgments.tex
\section{Acknowledgments}

This research is supported in part by the Trond Mohn Foundation project ''Algebraic and topological cycles in complex and tropical geometry'' and the Center for Advanced Study Young Fellows Project ''Real Structures in Discrete, Algebraic, Symplectic, and Tropical Geometries''.
We would like to thank Kris Shaw, Alex Fink, Andrew Berget, Lorenzo Vecchi and Benjamin Schröter for all the help with and interest in this project. We would also like to thank the reviewer of a previous version of this paper for the insight of \cref{rmk:minimal}.